\documentclass[12pt,reqno]{amsart}
\usepackage{amscd,amssymb,amsmath,amsthm}
\usepackage[arrow,matrix]{xy}
\usepackage{graphicx}
\usepackage{cite}
\usepackage{geometry}
\newtheorem{thm}{Theorem}

\newtheorem{lem}{Lemma}
\newtheorem{pro}{Proposition}
\newtheorem{rem}{Remark}
\newtheorem{cor}{Corollary}

\numberwithin{equation}{section} 

\newcommand{\bea}{\begin{eqnarray}}
\newcommand{\eea}{\end{eqnarray}}

\begin{document}
\title{On positive solutions  of the homogeneous Hammerstein integral equation}
\author{Yu. Kh. Eshkabilov, F.H. Haydarov}
%\date{\today}
\address{Yu.\ Kh.\ Eshkabilov\\ National University of Uzbekistan,
Tashkent, Uzbekistan.} \email {yusup62@mail.ru}

 \address{F.H.Haydarov\\ National University of Uzbekistan,
Tashkent, Uzbekistan.}
 \email {haydarov\_ imc@mail.ru.}

\begin{abstract} In this paper the existence and uniqueness positive
fixed points of the one nonlinear integral operator are discussed.
We prove that existence finite positive solutions  of the integral
equation of Hammerstein type. Obtained results applied to study
Gibbs measures for models on a Cayley tree.

\end{abstract}

\maketitle



\section{Introduction} \label{sec:intro}

It is well known that integral equations have wide applications in
engineering, mechanics, physics, economics, optimization,
vehicular traffic, biology, queuing theory and so on (see
\cite{gl1994},\cite{ko1972}, \cite{oh1996}, \cite{ab2002},
\cite{kz1984}). The theory of integral equations is rapidly
developing with the help of tools in functional analysis, topology
and fixed point theory. Therefore, many different methods are used
to obtain the solution of the nonlinear integral equation.
Moreover, some methods can be found in Refs. \cite{apz2000},
\cite{db2003}, \cite{a2003}, \cite{pm2004}, \cite{aed2005},
\cite{ak2006}, \cite{a2008}, \cite{aee2009}, to discuss and obtain
the solution of Hammerstein integral equation. In \cite{ak2006}
J.Appell and A.S. Kalitvin used fixed point methods and methods of
nonlinear spectral theory to obtain the solution of integral
equations of Hammerstein or Uryson type. The existence of positive
solutions of abstract integral equations of Hammerstein type is
discussed in  \cite{a2008}. In \cite{aee2009} M.A. Abdou, M.M.
El-Borai and M.M. El-Kojok the existence and uniqueness solution
of the nonlinear integral equation of Hammerstein type with
discontinuous kernel are discussed.

This present paper, we study solvability homogeneous  integral
equation of Hammerstein type. An integral equation of the form

\begin{equation}\label{e1.1}
\int_0^1K(t,u)\Psi(t,f(u))du=f(t)
\end{equation}
is called the homogeneous Hammerstein integral equation, where
$K(t,u)$ is continuous real-valued function defined on $0 \leq t
\leq 0, \,\ 0\leq u \leq 1, \,\
\Psi:[0,1]\times\mathbb{R}\rightarrow \mathbb{R}$  is a continuous
function and $f(t)$ is unknown function from $C[0,1]$.

Let $\Psi(t,z), \frac{\partial}{\partial z}\Psi(t,z)$ be
continuous and bounded for $t\in[0,1]$ and for all $z$. Then
\cite{pjc1997} the Hammerstein integral equation (\ref{e1.1}) has
a solution. Assume that  $\Psi(t,z)$ is a bounded continuous
function for $t\in[0,1]$ and $z\in\mathbb{R}$. In this case, also
the Hammerstein integral equation (\ref{e1.1}) has
\cite{k2003/2004} a solution. For the necessary details of this
theorem and for more results on the Hammerstein integral equation,
we refer to Petryshyn and Fitzpatrik \cite{pf1970}, Browder
\cite{b1971}, Brezis and Browder \cite{bb1975}.

Recently, in \cite{e2011} consider the case $\Psi(t,z)=\Psi(z)$.
Let $\Psi(z)$ is monotone left-continuous function on
$[0,+\infty)$ and $lim_{x\rightarrow0}\frac{\Psi(z)}{z}=+\infty,
\,\ lim_{x\rightarrow +\infty}\frac{\Psi(z)}{z}=0.$ Then
\cite{e2011} the integral equation of Hammerstein type
(\ref{e1.1}) has a solution.

In this work, we will consider the following integral equation of
Hammerstein type (i.e. in (\ref{e1.1})
$\Psi(t,z)=\Psi(z)=z^{\vartheta}$):

\begin{equation}\label{e1.2}
\int_0^1K(t,u)f^{\vartheta}(u)du=f(t), \,\ \vartheta>1
\end{equation}

on the $C[0,1]$, where $K(t,u)$ is strictly positive continuous
function.

By the Theorem 44.8 from \cite{kz1984} follows the existence of
nontrivial positive solution of the Hammerstein equation
(\ref{e1.2}). We study the problem of existence finite number
positive solutions of the integral equation of Hammerstein type
(\ref{e1.2}).

Consider the nonlinear operator $R_{\alpha}$ on the cone of
positive continuous functions on $[0,1]:$

\begin{equation}\label{e1.3}
\left(R_{\alpha}f\right)(t)=\left({\int_0^1K(t,u)f(u)du\over
\int_0^1 K(0,u)f(u)du}\right)^{\alpha},
\end{equation} \
where $K(t,u)$ is given in the integral equation of Hammerstein
type (\ref{e1.1}) and  $\alpha>0$. Operator of the form
(\ref{e1.3}) arising in the theory of Gibbs measures (see
\cite{er2010}, \cite{ehr2013}, \cite{u2013}). Positive fixed
points of the operator $R_{k}, \,\ k\in \mathbb{N}$ and they
numbers is very important to study Gibbs measures for models on a
Cayley tree.

In \cite{er2010}, in the case $\alpha=1$ the uniqueness positive
fixed points of the nonlinear operator $R_{\alpha}$  (\ref{e1.3})
is proved. In \cite{ehr2013}, in the case $\alpha=k\in \mathbb{N},
k>1$ for the nonlinear operator $R_{\alpha}$ was proved the
existence of positive fixed point and the existence Gibbs measure
for some mathematical models on a Cayley tree.

The aim of this work is to study the existence finite number
positive solutions of the Hammerstein equation (\ref{e1.2}) on the
space of continuous functions on $[0,1].$  The plan of this paper
is as follows. In the second section using properties of
Hammerstein equation (\ref{e1.2}) we reduce some statements on the
positive fixed point of the operator $R_{\alpha}$. In the third
section we construct the strictly positive continuous kernel
$K(t,u)$ such that, the corresponding Hammerstein equation
(\ref{e1.2}) has $n\in\mathbb{N}$ positive solutions. In the
fourth section obtained results for the operator $R_{\alpha}$
applied to study Gibbs measures for models on a Cayley tree.

\section{Existence and uniqueness of positive fixed points of the operator $R_{\alpha}$}

In this section we study the existence and the uniqueness positive
fixed points of the nonlinear operator $R_{\alpha}$ (\ref{e1.3}). Put\\
$$C^+[0,1]=\{f\in C[0,1]: f(x)\geq 0\}, \,\,\ C_0^+[0,1]=C^+[0,1]\setminus
\{\theta\equiv 0\}.$$

Well then the set $C^+[0,1]$ is the cone of positive continuous
functions on $[0,1].$

We define the Hammerstein operator $H_\vartheta$ on $C[0,1]$ by
the equality

$$H_\vartheta f(t)=\int_0^1K(t,u)f^{\vartheta}(u)du=f(t), \,\ \vartheta>1.$$

 Clearly, that by the Theorem 44.8 from \cite{kz1984} we obtained
\begin{thm}\label{t1} Let $\vartheta>1$. The equation
\begin{equation}\label{e2.1}
 H_\vartheta f=f
\end{equation}
 has at least one solution in $C_0^+[0,1].$
\end{thm}

Put
$$\mathcal M_0=\left\{f\in C^+[0,1]: f(0)=1\right\}.$$

\begin{lem}\label{l.1} Let $\alpha>1$. The equation
\begin{equation}\label{e2.2}
R_{\alpha}f=f, \,\ f\in C^{+}_{0}[0,1]
\end{equation}
 has a positive solution iff the Hammerstein operator
has a positive eigenvalue, i.e. the Hammerstein equation
\begin{equation}\label{e2.3}
H_{\alpha}g=\lambda g, \,\ f\in C^{+}[0,1]
\end{equation}
has a positive solution in $\mathcal M_0$ for some $\lambda>0$.
\end{lem}

\begin{proof} We define the linear operator $W$ and the linear functional
$\omega$ on the $C[0,1]$ by following equalities

$$(Wf)(t)=\int_0^1K(t,u)f(u)du, \,\,\ \omega(f)=\int_0^1K(0,u)f(u)du.$$

{\sl Necessariness.} Let $f_0\in C_0^+[0,1]$ be a solution of the
equation (\ref{e2.2}). We have
$$(Wf_0)(t)=\omega(f_0)\sqrt[\alpha]{f_0(t)}.$$
From this equality we get
$$(H_\alpha h)(t)=\lambda_0 h(t),$$
where $h(t)=\sqrt[\alpha]{f_0(t)}$ and $\lambda_0=\omega(f_0)>0$.

It is easy to see that $h\in \mathcal M_0$ and $h(t)$ is an
eigenfunction of the Hammerstein's operator $H_\alpha$,
corresponding the positive eigenvalue $\lambda_0$.

{\sl Sufficiency.}  Let $h\in \mathcal M_0$ be an eigenfunction of
the Hammerstein's operator $H_\alpha$. Then there is a number
$\lambda_0>0$ such that $H_\alpha h=\lambda_0 h$. From $h(0)=1$ we
get $\lambda_0=(H_\alpha h)(0)=\omega(h^\alpha)$. Then

$$h(t)={(H_\alpha h)(t)\over \omega(h^\alpha)}.$$
From this equality we get $R_\alpha f_0=f_0$ with $f_0=h^\alpha\in
C_0^+[0,1]$. This completes the proof. \end{proof}

\begin{thm}\label{t2} The equation (\ref{e2.2}) has at least one
solution in $C_0^+[0,1]$.
\end{thm}

Let $\lambda_{0}$ be a positive eigenvalue of the Hammerstein
operator $H_{\alpha}, \alpha>1.$ Then there exists $f_{0}\in
\mathcal M_0$ such that $H_{\alpha}f_{0}=\lambda_{0}f_{0}.$ Take
$\lambda\in (0,+\infty)$, $\lambda\ne\lambda_0$. Define function
$h_0(t)\in C_0^+[0,1]$ by
$$h_0(t)=\sqrt[\alpha-1]{\lambda\over \lambda_0}f_0(t), \ \ t\in [0,1].$$ Then
$$H_\alpha h_0=H_\alpha\left(\sqrt[\alpha-1]{\lambda\over \lambda_0}f_0\right)=\lambda h_0,$$

i.e. the number $\lambda$ is an eigenvalue of Hammerstein operator
$H_{\alpha}$ corresponding the eigenfunction $h_0(t)$. It can be
easily checked: if the number $\lambda_{0}>0$ is eigenvalue of the
operator $H_\alpha, \alpha>1$, then an arbitrary positive number
is eigenvalue of the operator $H_\alpha$. Therefore we have

\begin{lem}\label{1.2} $a)$Let $\alpha>1$.The equation  $R_{\alpha}f=f$ has a nontrivial
positive solution iff the Hammerstein equation $H_{\alpha}g=g$ has
a nontrivial positive solution.
\end{lem}

Let $\alpha>1$. Denote by $N_{fix.p}(H_\alpha)$ and
$N_{fix.p}(R_\alpha)$ numbers of nontrivial positive solutions of
the equations (\ref{e2.1}) and (\ref{e2.2}), respectively.

\begin{thm}\label{t3} Let $\alpha>1$. The equality $N_{fix.p}(H_\alpha)=N_{fix.p}(R_\alpha)$ is
held.
\end{thm}

Denote
$$m=\min_{t,u\in [0,1]}K(t,u), \ \ M_0=\max_{u\in [0,1]}K(0,u),$$
$$M=\max_{t,u\in [0,1]}K(t,u), \ \ m_0=\min_{u\in [0,1]}K(0,u).$$

\begin{thm}\label{t4} Let $\alpha>1.$ If the following inequality
holds

$$\left(\frac{M}{m_0}\right)^{\alpha}-\left(\frac{m}{M_0}\right)^{\alpha}<\frac{1}{\alpha}$$\\

then the homogenous Hammerstein equation (\ref{e2.1}) and the
equation (\ref{e2.2}) has unique nontrivial positive solution.
\end{thm}

Analogous theorem was proved for $\alpha=k\in\mathbb{N}, k\geq2$
in \cite{ehr2013} and proof of the Theorem \ref{t4} is analogously
to its.

\section{Existence finite positive solutions of homogeneous  Hammerstein equation}

In this section we'll show the existence of $n\in \mathbb{N}$
positive solutions of homogeneous integral equation of Hammerstein
type (\ref{e1.2}).

For all $p,n\in \mathbb{N}$ we define following matrices:

 \begin{equation}\label{e1}{\mathbf A}_{n}^{(p)}=\left\{\frac{1}{2(2p+i+j)-3}\left(\frac{1}{2}\right)^{2(2p+i+j-2)}
\right\}_{i,j=\overline{1,n}}, \,\ n,p\in
\mathbb{N}.\end{equation}

\begin{equation}\label{e2}
{\mathbf
B}[a_{1},...,a_{n};b_{1},...b_{n}]=\left(\frac{1}{a_{i}+b_{j}}\right)_{i,j=\overline{1,n}},
\,\ a_{i}, b_{j}>0.\end{equation}

\begin{equation}\label{e3}
{\mathbf
C}_{n}^{(p)}=B[4p,4(p+1),...,2(p+n-1);1,5,...,4n-3].\end{equation}

\begin{lem}\label{l3}\cite{pr1978} Let $n\geq 2.$ Then

 $$\det{\mathbf B}[a_{1},...,a_{n};b_{1},...,b_{n}]=\frac{\prod_{1\leq i<j\leq n}[(a_{i}-a_{j})
 (b_{i}-b_{j})]}{\prod_{i,j=1}^{n}(a_{i}+b_{j})}$$
\end{lem}

 \begin{cor}\label{c1}
 $\det {\mathbf A}_{n}^{(p)}=\left(\frac{1}{2}\right)^{2n(2p+n-1)}\det {\mathbf C}_{n}^{(p)}.$
 \end{cor}

 \begin{proof} Let $i,j=\overline{1,n}.$ We multiply by
 $2^{2(p+j-1)}$ the column $j$ of the matrix ${\mathbf A}_{n}^{(p)}$ and
 then multiply by $2^{2(i-1)}$ the row $i$ of the matrix obtained.
 As a result we get ${\mathbf C}_{n}^{(p)}.$
\end{proof}

 \begin{lem}\label{l4} Let
 ${\mathbf B}^{-1}[a_{1},a_{2},...,a_{n};b_{1},b_{2},...,b_{n}]=\{\beta_{ij}\}_{i,j=\overline{1,n}}$
 \ is an inverse matrix of\\
 ${\mathbf B}[a_{1},a_{2},...,a_{n};b_{1},b_{2},...,b_{n}].$ Then

  $$\beta_{ji}=\frac{\prod_{s=1}^{n}(a_{s}+b_{j})\prod_{s=1,s\neq i}^{n}(a_{i}+b_{s})}
  {\prod_{s=1,s\neq j}^{n}(b_{j}-b_{s})\prod_{s=1,s\neq i}^{n}(a_{i}-a_{s})}$$
\end{lem}

  \begin{proof} Subtracting the $j$th column of
  ${\mathbf B}[a_{1},a_{2},...,a_{n};b_{1},b_{2},...,b_{n}]$ from every
  other column we get a following equality

   $$\det{\mathbf B}[a_{1},a_{2},...,a_{n};b_{1},b_{2},...,b_{n}]=$$

   $$=\frac{\prod_{s=1,s\neq
   j}^{n}(b_{j}-b_{s})}{\prod_{s=1}^{n}(a_{s}+b_{j})} \left(
\begin{array}{ccccccc}
\frac{1}{a_{1}+b_{1}}&...&\frac{1}{a_{1}+b_{j-1}}&1&\frac{1}{a_{1}+b_{j+1}}&...&\frac{1}{a_{1}+b_{n}}\\[3mm]
\frac{1}{a_{2}+b_{1}}&...&\frac{1}{a_{2}+b_{j-1}}&1&\frac{1}{a_{2}+b_{j+1}}&...&\frac{1}{a_{2}+b_{n}}\\[3mm]
&...& &...& &...&\\[3mm]
\frac{1}{a_{n}+b_{1}}&...&\frac{1}{a_{n}+b_{j-1}}&1&\frac{1}{a_{n}+b_{j+1}}&...&\frac{1}{a_{n}+b_{n}}\\[3mm]
\end{array}
\right).$$\\
 Now we subtract from the $j$th row the $i$th row for every $j\in\{1,2...,i-1,i+1,...n\}.$
  Then

  $$\det {\mathbf B}[a_{1},...,a_{n};b_{1},...,b_{n}]=\frac{\prod_{s=1,s\neq j}^{n}(b_{j}-b_{s})\prod_{s=1,s\neq i}^{n}(a_{i}-a_{s})}
  {\prod_{s=1}^{n}(a_{s}+b_{j})\prod_{s=1,s\neq i}^{n}(a_{i}+b_{s})} \ \times\ {\det\mathbf
  B}^{(i,j)}[a_{1},...,a_{n};b_{1},...,b_{n}]$$\\

where ${\mathbf B}^{(j,i)}[a_{1},...,a_{n};b_{1},...,b_{n}]$ is
the cofactor of the element $\frac{1}{a_{i}+a_{j}}$ in ${\mathbf
B}[a_{1},...,a_{n};b_{1},...,b_{n}].$

Since $$\beta_{ji}=\frac{{\det\mathbf
B}^{(i,j)}[a_{1},...,a_{n};b_{1},...,b_{n}]}{\det {\mathbf
B}[a_{1},...,a_{n};b_{1},...,b_{n}]}.$$

This completes the proof.
\end{proof}
 Let be

$$({\mathbf A}_{n}^{(p)})^{-1}=\{\alpha_{ij}\}_{i,j\in\overline{1,n}}.$$

\begin{rem}\label{r1} For each $\alpha_{ji}$ element of $({\mathbf A}_{n}^{(p)})^{-1}$ the following equality holds

$$\alpha_{ji}=4^{2p+i+j-n+1}\cdot \frac{\prod_{s=1}^{n}(4p+2s+2j-3)\prod_{s=1,s\neq j}^{n}(4p+2s+2j-3)}
{\prod_{s=1,s\neq j}^{n}(j-s)\prod_{s=1,s\neq i}^{n}(i-s)}$$

\end{rem}

\begin{proof} By Corollary \ref{c1} and Lemma \ref{l4} we get

 $$\alpha_{ji}=4^{2p+i+j}\cdot\frac{\det{\mathbf B}^{(i,j)}[4p,4p+2,...4p+2(n-1);1,3,...2n-1]}{\det{\mathbf B}[4p,4p+2,...4p+2(n-1);1,3,...2n-1]}=$$

$$=4^{2p+i+j}\cdot\frac{\prod_{s=1}^{n}(4p+2s+2j-3)\prod_{s=1,s\neq i}^{n}(4p+2s+2i-3)}
{\prod_{s=1,s\neq j}^{n}(2j-2s)\prod_{s=1,s\neq i}^{n}(2i-2s)}.$$

\end{proof}

Denote

$$\varphi_{(s,n,p)}(u)=\alpha_{s1}u^{2p-1}+...+\alpha_{sn}u^{2(n+p)-3},\ \ s,n,p\in \mathbb{N}, \  u\in [0,1].$$

$$K_{(n,p)}(t,u;k)=1+\sum_{s=1}^{n}\left(\sqrt[k]{1+t^{2(p+s)-1}}-1\right)\varphi_{(s,n,p)}(u), \ \ k\in \mathbb{N},k\geq 2, \  t,u\in [0,1]. $$

\begin{rem}\label{r2} For the given $k\in \mathbb{N}, k\geq 2$ the following
inequality holds
$$K_{(n,p)}\left(t-\frac{1}{2}, u-\frac{1}{2}; k\right)\leq K_{(n,1)}\left(t-\frac{1}{2}, u-\frac{1}{2}; k\right), \ \ (t,u)\in [0,1]^{2}, n, p\in \mathbb{N}.$$
\end{rem}
Put
$$\zeta_0(n)=\frac{64}{9}\cdot\frac{4^{n}-1}{4n+1}\left(\frac{(4n+1)!!}{(n-1)!(2n+1)!!}\right)^{2}.$$\\

\begin{lem}\label{lm1} Let $n\in \mathbb{N}$. If  $k\geq \zeta_0(n)$
then the following inequality holds

$$K_{(n,p)}\left(t-\frac{1}{2},u-\frac{1}{2};k\right)>0, \ (t,u)\in[0,1]^{2}, p\in \mathbb{N}.$$
\end{lem}
\begin{proof} For $p=1$ from Remark \ref{r1} we
have\vskip0.3truecm
$$\alpha_{ij}=4^{i+j-n+3}\frac{\prod_{s=1}^{n}(2i+2s+1)\prod_{s=1,s\neq j}^{n}(2j+2s+1)}
{\prod_{s=1,s\neq i}^{n}(i-s)\prod_{s=1,s\neq
j}^{n}(j-s)}.$$\vskip0.3truecm
Then\\
$$\left|\frac{\alpha_{i,j+1}}{\alpha_{i,j}}\right|=\frac{4(4j+1)(2j+2n+3)}{(2j+3)(4j+5)},\ i=\overline{1,n},\ j=\overline{1,n-1}$$
and \vskip0.1truecm
$$\left|\frac{\alpha_{i+1,j}}{\alpha_{i,j}}\right|=\frac{4(n-i)(2i+2n+3)}{i(2i+3)},\
i=\overline{1,n},\
j=\overline{1,n-1}.$$\\
From above one has
$\max_{i,j=\overline{1,n}}|\alpha_{ij}|=|\alpha_{nn}|.$
By Remark \ref{r1} we can take\\
$$K_{(n,p)}\left(t-\frac{1}{2}, u-\frac{1}{2}; k\right)\geq 1- \frac{2}{3}\max_{i,j=\overline{1,n}}|\alpha_{ij}|
 \sum_{s=1}^{n}\left(\sqrt[k]{1+\left(\frac{1}{2}\right)^{2s+1}-1}\right)\geq$$\vskip0.3truecm
$$\geq 1- \frac{2|a_{nn}|}{3k}\sum_{s=1}^{n}\left(\frac{1}{2}\right)^{2s+1}\geq
1-(4^{n}-1)\cdot
\frac{64(2n+3)^{2}(2n+5)^{2}...(4n-1)^{2}(4n+1)}{9k
\left((n-1)!\right)^{2}}.$$ \vskip0.3truecm Since $k\geq
\zeta_0(n)$ one get
$K_{(n,p)}\left(t-\frac{1}{2},u-\frac{1}{2};k\right)>0.$ This
completes the proof.

\end{proof}

\begin{pro}\label{pr4} Let $n\in \mathbb{N}$. If  $k\geq \zeta_0(n)$ then the
 Hammerstein's nonlinear operator $H_{k}$ with the kernel
  $K_{(n,p)}\left(t-\frac{1}{2},u-\frac{1}{2};k\right) \,\,\ (p\in \mathbb{N})$ has at least n positive fixed points.
\end{pro}
\begin{proof} Let $f_{j}(u)=\sqrt[k]{1+u^{2(p+j)-1}},\
j=\overline{1,n}$ and $u_{1}=u-\frac{1}{2}, t_{1}=t-\frac{1}{2}.$
Put $g_j(t)=f_{j}(t-\frac{1}{2}).$ We are showing functions
$g_j(t)$ are fixed points of the Hammerstein operator $H_{k}$ with
the kernel
  $K_{(n,p)}\left(t-\frac{1}{2},u-\frac{1}{2};k\right):$

$$\int_{0}^{1}K_{(n,p)}\left(t-\frac{1}{2},u-\frac{1}{2};k\right)g_{j}^{k}(u)du=$$
$$=\int_{0}^{1}K_{(n,p)}\left(t-\frac{1}{2},u-\frac{1}{2};k\right)f_{j}^{k}\left(u-\frac{1}{2}\right)du=
  \int_{-\frac{1}{2}}^{\frac{1}{2}}K_{(n,p)}\left(t_{1},u_{1};k\right)f_{j}^{k}(u_{1})du_{1}=$$

$$\int_{\frac{1}{2}}^{-\frac{1}{2}}\left[1+\sum_{s=1}^{n}\left(\sqrt[k]{1+t_{1}^{2(p+s)-1}}-1\right)
\varphi_{(s,n,p)}(u_{1})\right]\left(1+u_{1}^{2(p+j)-1}\right)du_{1}=$$

 $$1+\sum_{s=1}^{n}\left(\sqrt[k]{1+t_{1}^{2(p+s)-1}}-1\right)
 \int_{\frac{1}{2}}^{\frac{1}{2}}\left(\alpha_{s1}u_{1}^{4(p-1)+2s+2j}+...+
 \alpha_{sn}u_{1}^{4p+2(s+j+n)-6}\right)du_{1}=$$

 $$1+\sum_{s=1}^{n}\left(\sqrt[k]{1+t_{1}^{2(p+s)-1}}-1\right)\left(\alpha_{sj}\beta_{s1}+...+\alpha_{nj}\beta_{sn}\right)=
 \sqrt[k]{1+t_{1}^{2(p+j)-1}}.$$ Hence

 $$\int_{0}^{1}K_{(n,p)}\left(t-\frac{1}{2},u-\frac{1}{2};k\right)g_{j}^{k}(u)du=g_{j}(t), \,\,\ j\in \{1,2,..., n\}.$$
\end{proof}

\begin{thm}\label{t5} For each $n\in \mathbb{N}$ there exists $\vartheta>1$ and a positive
continuous kernel $K(t,u)$ such that Hammerstein integral equation
(\ref{e1.2}) has n positive solutions.
\end{thm}

\section{Gibbs measures for models on Cayley tree $\Gamma^k$}

In this section we study Gibbs measures for models on Cayley tree.
You may be acquaint with definitions and properties of Gibbs
measures in books \cite{g2011}, \cite{p1974}, \cite{u2013}. A
Cayley tree (Bethe lattice) $\Gamma^k$ of order $k\in \mathbb{N}$
is an infinite homogeneous tree, i.e., a graph without cycles,
such that exactly $k+1$ edges originate from each vertex. Let
$\Gamma^k=(V,L)$ where $V$ is the set of vertices and $L$ that of
edges (arcs). Two vertices $x$ and $y$ are called nearest
neighbors if there exists an edge $l\in L$ connecting them. We
will use the notation $l=\langle x,y\rangle$. A collection of
nearest neighbor pairs $\langle x,x_{1}\rangle ,\langle
x_{1},x_{2}\rangle,...\langle x_{d-1},y\rangle$ is called a $path$
from $x$ to $y$. The distance $d(x,y)$ on the Cayley tree is the
number of edges of the shortest path from $x$ to $y$.

 For a fixed
$x^{0}\in V$, called the root, we set
$$W_{n}=\{x\in V| d(x,x^{0})=n\}, \,\,\,\,\,\,\,\
 V_{n}=\bigcup_{m=0}^{n}W_{m}$$
  and denote
  $$S(x)=\{y\in W_{n+1}: d(x,y)=1\}, x\in W_{n},$$
  the set of $direct$  $successors$ of $x$.

Consider models where the spin takes values in the set $[0,1]$,
and is assigned to the vertexes of the tree. For $A\subset V$ a
configuration $\sigma_A$ on $A$ is an arbitrary function
$\sigma_A:A\to [0,1]$. Denote $\Omega_A=[0,1]^A$ the set of all
configurations on $A$ and $\Omega=[0,1]^V$. The Hamiltonian on
$\Gamma^k$ of the model is

\begin{equation}\label{2.1} H(\sigma)=-J\sum_{\langle x,y\rangle\in L}
\xi\left(\sigma(x),\sigma(y)\right), \,\,\ \sigma\in\Omega \end{equation}\\
 where $J \in R\setminus \{0\}$ and
$\xi: (u,v)\in [0,1]^2\to \xi_{u,v}\in \mathbb{R}$ is a given
bounded, measurable function.

Let $\lambda$ be the Lebesgue measure on $[0,1]$.  On the set of
all configurations on $A$ the a priori measure $\lambda_A$ is
introduced as the $|A|$ fold product of the measure $\lambda $.
Here and further on $|A|$ denotes the cardinality of $A$.   We
consider a standard sigma-algebra ${\mathcal B}$ of subsets of
$\Omega=[0,1]^V$ generated by the measurable cylinder subsets.

Let $\sigma_n:\;x\in V_n\mapsto \sigma_n(x)$ is a configuration in
$V_n$ and $h:\;x\in V\mapsto h_x=(h_{t,x}, t\in [0,1]) \in
\mathbb{R}^{[0,1]}$ be mapping of $x\in V\setminus \{x^0\}$.
Given $n=1,2,\ldots$, consider the probability distribution
$\mu^{(n)}$ on $\Omega_{V_n}$ defined by

\begin{equation}\label{2.2}\mu^{(n)}(\sigma_n)=Z_n^{-1}\exp\left(-\beta H(\sigma_n)
+\sum_{x\in W_n}h_{\sigma(x),x}\right).\end{equation}\\
 Here, as before, $\sigma_n:x\in V_n\mapsto
\sigma(x)$ and $Z_n$ is the corresponding partition function:

\begin{equation}\label{2.3}Z_n=\int_{\Omega_{V_n}} \exp\left(-\beta H({\widetilde\sigma}_n)
+\sum_{x\in W_n}h_{{\widetilde\sigma}(x),x}\right)
\lambda_{V_n}(\widetilde{\sigma}_{n}),\end{equation}\\

 where $\beta=T^{-1}, T>0 -$ temperature. The probability distributions $\mu^{(n)}$ are compatible \cite{er2010} if for any
$n\geq 1$ and $\sigma_{n-1}\in\Omega_{V_{n-1}}$:

\begin{equation}\label{2.4}\int_{\Omega_{W_n}}\mu^{(n)}(\sigma_{n-1}\vee\omega_n)\lambda_{W_n}(d(\omega_n))=
\mu^{(n-1)}(\sigma_{n-1}). \end{equation}\\
 Here
$\sigma_{n-1}\vee\omega_n\in\Omega_{V_n}$ is the concatenation of
$\sigma_{n-1}$ and $\omega_n$. In this case there exists
\cite{er2010} a unique measure $\mu$ on $\Omega_V$ such that, for
any $n$ and $\sigma_n\in\Omega_{V_n}$, $\mu \left(\left\{\sigma
\Big|_{V_n}=\sigma_n\right\}\right)=\mu^{(n)}(\sigma_n)$.\\
The measure $\mu$ is called {\it splitting Gibbs measure}
corresponding to Hamiltonian (\ref{2.1}) and function $x\mapsto
h_x$, $x\neq x^0$.

The following statement describes conditions on $h_x$ guaranteeing
compatibility
of the corresponding distributions $\mu^{(n)}(\sigma_n).$\\
\begin{pro}\label{pr5}\cite{er2010} The probability distributions
$\mu^{(n)}(\sigma_n)$, $n=1,2,\ldots$, in (\ref{2.2}) {\sl are
compatible iff for any $x\in V\setminus\{x^0\}$ the following
equation holds:

\begin{equation}\label{2.5} f(t,x)=\prod_{y\in S(x)}{\int_0^1\exp(J\beta\xi_{t,u})f(u,y)du
\over \int_0^1\exp(J\beta{\xi_{0,u}})f(u,y)du}. \end{equation}\\
 Here,
and below $f(t,x)=\exp(h_{t,x}-h_{0,x}), \ t\in [0,1]$ and
$du=\lambda(du)$ is the Lebesgue measure.}
\end{pro}

We consider $\xi_{tu}$ as a continuous function and we are going
to solve equation (\ref{2.5}) in the class of
$translation-invariant$ functions $f(t,x)$ (i.e. $f(t,x)=f(t)$ for
all $x\in \Gamma^k\setminus \{x_0\}$) and we'll show that there
exists a finite number of $translation-invariant$ Gibbs measures
for model (\ref{2.1}).

Then for $translation$-$invariant$ functions the equation
(\ref{2.5}) can be written as
 \begin{equation}\label{3.90} (R_{k}f)(t)=f(t), \ k\in \mathbb{N} \end{equation}\\

 where $K(t,u)=Q(t,u)=\exp(J\beta \xi_{tu}),\,\
f(t)\in C^+_{0}[0,1],\,\ t,u\in [0,1]$ (see \cite{er2010},\cite{ehr2013}).\\

Consequently, for each $k\in \mathbb{N}, k\geq2$ Hammerstein
integral equation corresponding to the equation (\ref{3.90})  has
the following form:

  \begin{equation}\label{4.1} \int^{1}_{0}Q(t,u)f^{k}(u)du=f(t).
  \end{equation}\\

By Theorem \ref{t3} and Propositions \ref{pr4}, \ref{pr5} we'll
get following Theorem.

\begin{thm}\label{t6} Let $n\in \mathbb{N}$. If $k\geq \zeta_0(n)$ then the model

$$H(\sigma)=-\frac{1}{\beta}\sum_{<x,y>}ln\left(K_{(n,p)}
 \left(\sigma(x)-\frac{1}{2},
 \sigma(y)-\frac{1}{2};k\right)\right), \,\,\ \sigma\in\Omega, (p\in \mathbb{N})$$\\
 on the Cayley tree $\Gamma^{k}$ has at least $n$ translation-invariant Gibbs
 measures.\end{thm}

\end{document}